\documentclass[11pt,fleqn]{article}

\usepackage{amsmath,amssymb,amsthm}
\usepackage[margin=1in]{geometry}
\usepackage[pdfpagemode=UseNone,pdfstartview=FitH]{hyperref}

\newcommand{\as}{\ins{as}}
\newcommand{\bdry}[1]{\partial #1}
\newcommand{\bgset}[1]{\big\{#1\big\}}
\newcommand{\closure}[1]{\overline{#1}}
\newcommand{\dualp}[3][]{\left(#2,#3\right)_{#1}}
\newcommand{\hquad}{\hspace{0.08in}}
\newcommand{\incl}{\subset}
\newcommand{\ins}[1]{\hquad \text{#1} \hquad}
\newcommand{\isom}{\approx}
\newcommand{\norm}[2][]{\left\|#2\right\|_{#1}}
\renewcommand{\o}{\text{o}}
\newcommand{\PS}[1]{$(\text{PS})_{#1}$}
\newcommand{\QED}{\mbox{\qedhere}}

\newcommand{\restr}[2]{\left.#1\right|_{#2}}

\newcommand{\set}[1]{\left\{#1\right\}}

\newcommand{\F}{{\cal F}}
\newcommand{\M}{{\cal M}}

\newcommand{\R}{\mathbb R}
\newcommand{\RP}{\R \text{P}}
\newcommand{\Z}{\mathbb Z}

\DeclareMathOperator{\dvg}{div}

\newenvironment{enumroman}{\begin{enumerate}

}{\end{enumerate}}

\newtheorem{corollary}{Corollary}[section]
\newtheorem{lemma}[corollary]{Lemma}
\newtheorem{proposition}[corollary]{Proposition}
\newtheorem{theorem}[corollary]{Theorem}

\theoremstyle{definition}
\newtheorem{definition}[corollary]{Definition}

\numberwithin{equation}{section}

\title{\bf An existence result for a class of quasilinear elliptic eigenvalue problems in unbounded domains\thanks{{\em MSC2010:} Primary 35J66, Secondary 35J70, 35J20 \smallskip
\newline \indent\; {\em Key Words and Phrases:} nonlinear eigenvalue problems, unbounded domains, weighted $p$-Laplacian, Robin boundary conditions, nontrivial solutions, Morse theory, cohomological local splitting}}
\author{\bf Kanishka Perera\\
Department of Mathematical Sciences\\
Florida Institute of Technology\\
Melbourne, FL 32901, USA\\
\em kperera@fit.edu\\
[\bigskipamount]
\bf Patrizia Pucci\\
Dipartimento di Matematica e Informatica\\
Universit\`a degli Studi di Perugia\\
Via Vanvitelli 1, 06123 Perugia, Italy\\
\em pucci@dmi.unipg.it\\
[\bigskipamount]
\bf Csaba Varga\\
Faculty of Mathematics and Computer Science\\
Babe\c{s}--Bolyai University\\
400084 Cluj--Napoca, Romania\\
\em csvarga@cs.ubbcluj.ro}

\begin{document}

\maketitle

\begin{abstract}
We consider a nonlinear eigenvalue problem under Robin boundary conditions in a domain with (possibly noncompact) smooth boundary. The problem involves a weighted $p$-Laplacian operator and subcritical nonlinearities satisfying Ambrosetti-Rabinowitz type conditions. Using Morse theory and a cohomological local splitting as in Degiovanni et al.\! \cite{MR2661274}, we prove the existence of a nontrivial weak solution for all (real) values of the eigenvalue parameter. Our result is new even in the semilinear case $p = 2$ and complements some recent results obtained in Autuori et al.\! \cite{AuPuVa}.
\end{abstract}

\newpage

\section{Introduction}

Let $\Omega$ be a domain in $\R^N,\, N \ge 2$ with (possibly noncompact) smooth boundary $\bdry{\Omega}$ and consider the problem
\begin{equation} \label{1.1}
\left\{\begin{aligned}
&- \dvg \left(a(x)\, |\nabla u|^{p-2}\, \nabla u\right) = \lambda\, f(x)\, |u|^{p-2}\, u + g(x,u) && \text{in } \Omega\\[10pt]
&a(x)\, |\nabla u|^{p-2}\, {\partial_\nu u} + b(x)\, |u|^{p-2}\, u = h(x,u) && \text{on } \bdry{\Omega},
\end{aligned}\right.
\end{equation}
where $1 < p < N$ and ${\partial_\nu u}=\partial/\partial \nu$ is the exterior normal derivative on $\bdry{\Omega}$. We assume that \linebreak {\em $a \in C^1(\Omega) \cap L^\infty(\Omega)$, with $a \ge a_0$ for some constant $a_0 > 0$, that $b \in C(\bdry{\Omega})$ verifies
\[
\frac{1}{C}\ (1 + |x|)^{1-p} \le b(x) \le C\, (1 + |x|)^{1-p} \quad \mbox{for all} x \in \bdry{\Omega},
\]
and that $f$ is a measurable function on $\Omega$ satisfying
\[
0 < f(x) \le C\, w_1(x) \quad \text{for a.a. } x \in \Omega,
\]
where $w_1(x) = (1 + |x|)^{- \alpha_1},\, p < \alpha_1 < N$}. Here and in the sequel $C$ denotes a generic positive constant.

If $w$ is a weight on $\Omega$ (i.e., an a.e.\! positive measurable function), let $L^\sigma(\Omega,w)$, $\sigma \ge 1$, denote the weighted Lebesgue space with the norm
\[
\norm[\sigma,w]{u} = \left(\int_\Omega w(x)\, |u(x)|^\sigma\, dx\right)^{1/\sigma}.
\]
Similarly, if $\widetilde{w}$ is a weight on $\bdry{\Omega}$, $L^\sigma(\bdry{\Omega},\widetilde{w})$ denotes the weighted space with the norm
\[
\norm[\sigma,\widetilde{w},\bdry{\Omega}]{u} = \left(\int_{\bdry{\Omega}} \widetilde{w}(x)\, |u(x)|^\sigma\, dS\right)^{1/\sigma}.
\]
We assume that
{\em\begin{itemize}
\item[$(g_1)$] $g$ is a Carath\'{e}odory function on $\closure{\Omega} \times \R$ satisfying
    \[
    |g(x,s)| \le g_0(x)\, |s|^{p-1} + g_1(x)\, |s|^{r-1} \quad \text{for a.a. } x \in \Omega \text{ and all } s \in \R,
    \]
    where $p < r < p^\ast := Np/(N - p)$, $g_0$ and $g_1$ are measurable functions on $\Omega$ satisfying
    \[
    0 < g_0(x) \le C\, w_2(x), \quad 0 \le g_1(x) \le C\, g_0(x) \quad \text{for a.a. } x \in \Omega,
    \]
    $w_2(x) = (1 + |x|)^{- \alpha_2},\, N - (N - p)\, r/p < \alpha_2 < N$, $g_0/w_2 \in L^{r/(r-p)}(\Omega,w_2)$, and \linebreak $g_0 \in L^{\widetilde{r}/(\widetilde{r}-p)}(\Omega,w_2)$, $r < \widetilde{r} < \min \set{pr,p^\ast}$;
\item[$(g_2)$] $G(x,s) := \int_0^s g(x,t)\, dt$ satisfies
    \[
    \lim_{s \to 0}\, \frac{G(x,s)}{g_0(x)\, |s|^p} = 0, \quad \lim_{|s| \to \infty}\, \frac{G(x,s)}{g_0(x)\, |s|^p} = \infty \quad \text{uniformly for a.a. } x \in \Omega;
    \]
\item[$(g_3)$] there exists a $\mu > p$ such that for a.a. $x \in \Omega$ and all $s \in \R$
    \[
    \mu\, G(x,s) \le s\, g(x,s) + \gamma_0(x) + \gamma_1(x)\, |s|^p,
    \]
    where $\gamma_0 \in L^1(\Omega)$ and $\gamma_1$ is a measurable function on $\Omega$ satisfying
    \[
    0 < \gamma_1(x) \le C\, w_1(x) \quad \text{for a.a. } x \in \Omega;
    \]
\item[$(h_1)$] $h$ is a Carath\'{e}odory function on $\bdry{\Omega} \times \R$ satisfying
    \[
    |h(x,s)| \le h_0(x)\, |s|^{p-1} + h_1(x)\, |s|^{q-1} \quad \text{for a.a. } x \in \bdry{\Omega} \text{ and all } s \in \R,
    \]
    where $p < q < \bar{p} := (N - 1)\, p/(N - p)$, $h_0$ and $h_1$ are measurable functions on $\bdry{\Omega}$ satisfying
    \[
    0 < h_0(x) \le C\, w_3(x), \quad 0 \le h_1(x) \le C\, h_0(x) \quad \text{for a.a. } x \in \bdry{\Omega},
    \]
    $w_3(x) = (1 + |x|)^{- \alpha_3},\, N - 1 - (N - p)\, q/p < \alpha_3 < N$, $h_0/w_3 \in L^{q/(q-p)}(\bdry{\Omega},w_3)$, and $h_0 \in L^{\widetilde{q}/(\widetilde{q}-p)}(\bdry{\Omega},w_3),\, q < \widetilde{q} < \min \set{qr,\bar{p}}$;
\item[$(h_2)$] $H(x,s) := \int_0^s h(x,t)\, dt$ satisfies
    \[
    \lim_{s \to 0}\, \frac{H(x,s)}{h_0(x)\, |s|^p} = 0, \quad \lim_{|s| \to \infty}\, \frac{H(x,s)}{h_0(x)\, |s|^p} = \infty \quad \text{uniformly for a.a. } x \in \bdry{\Omega};
    \]
\item[$(h_3)$] there exists a $\widetilde{\mu} > p$ such that for a.a. $x \in \bdry{\Omega}$ and all $s \in \R$
    \[
    \widetilde{\mu}\, H(x,s) \le s\, h(x,s) + \tau_0(x) + \tau_1(x)\, |s|^p,
    \]
    where $\tau_0 \in L^1(\bdry{\Omega})$, $\tau_1$ is a measurable functions on $\bdry{\Omega}$ satisfying
    \[
    0 < \tau_1(x) \le C\, w_4(x) \quad \text{for a.a. } x \in \bdry{\Omega},
    \]
    and $w_4(x) = (1 + |x|)^{- \alpha_4},\, p - 1 < \alpha_4 < N$.
\end{itemize}}

Let $C^\infty_\delta(\Omega)$ be the space of $C^\infty_0(\R^N)$-functions restricted to $\Omega$, and define the weighted Sobolev space $E$ to be the completion of $C^\infty_\delta(\Omega)$ with respect to the norm
\[
\norm[E]{u} = \left(\int_\Omega |\nabla u(x)|^p + \frac{|u(x)|^p}{(1 + |x|)^p}\ dx\right)^{1/p}.
\]
By Pfl{\"u}ger \cite[Lemma 2]{MR1615337},
\[
\norm{u} = \left(\int_\Omega a(x)\, |\nabla u(x)|^p\, dx + \int_{\bdry{\Omega}} b(x)\, |u(x)|^p\, dS\right)^{1/p}
\]
is an equivalent norm on $E$. A {\em weak solution} of problem \eqref{1.1} is a function $u \in E$ satisfying
\begin{multline*}
\int_\Omega a(x)\, |\nabla u|^{p-2}\, \nabla u \cdot \nabla \varphi\, dx + \int_{\bdry{\Omega}} b(x)\, |u|^{p-2}\, u\, \varphi\, dS = \int_\Omega \left(\lambda\, f(x)\, |u|^{p-2}\, u + g(x,u)\right) \varphi\, dx\\[10pt]
+ \int_{\bdry{\Omega}} h(x,u)\, \varphi\, dS
\end{multline*}
for all $\varphi \in E$.

Since $g(x,0) = 0$ for a.a.\! $x \in \Omega$ by $(g_1)$ and $h(x,0) = 0$ for a.a.\! $x \in \bdry{\Omega}$ by $(h_1)$, \eqref{1.1} has the trivial solution $u = 0$. The following existence result was recently obtained by Autuori et al.\! \cite[Theorem 4.3]{AuPuVa}.

\begin{theorem} \label{Theorem 1.1} In addition to $(g_1)$--$(g_3)$ and $(h_1)$--$(h_3)$,
assume that
\[
G(x,s) \ge 0 \text{ for a.a. } x \in \Omega \text{ and all } s \in \R
\text{ and }
H(x,s) \ge 0 \text{ for a.a. } x \in \bdry{\Omega} \text{ and all } s \in \R.
\]
Then problem \eqref{1.1} has a nontrivial solution $u \in E$ for all $\lambda \in \R$.
\end{theorem}

The proof of this theorem was based on the sequence of eigenvalues $\lambda_k \nearrow \infty$ of the associated Robin boundary eigenvalue problem
\begin{equation} \label{1.2}
\left\{\begin{aligned}
&- \dvg \left(a(x)\, |\nabla u|^{p-2}\, \nabla u\right) = \lambda\, f(x)\, |u|^{p-2}\, u && \text{in } \Omega,\\[10pt]
&a(x)\, |\nabla u|^{p-2}\, \frac{\partial u}{\partial \nu} + b(x)\, |u|^{p-2}\, u = 0 && \text{on } \bdry{\Omega},
\end{aligned}\right.
\end{equation}
defined using the $\Z_2$-cohomological index. In the present paper we obtain the following extension.

\begin{theorem} \label{Theorem 1.2}
Assume $(g_1)$--$(g_3)$ and $(h_1)$--$(h_3)$. Then problem \eqref{1.1} has a nontrivial solution $u \in E$ in each of the following cases:
\begin{enumroman}
\item \label{Theorem 1.2.i} $\lambda \not\in \bgset{\lambda_k : k \ge 1}$;
\item \label{Theorem 1.2.ii} $\lambda \in \bgset{\lambda_k : k \ge 1}$, $G(x,s) \ge 0$ for a.a. $x \in \Omega$ and all $s \in \R$, and $H(x,s) \ge 0$ for a.a. $x \in \bdry{\Omega}$ and all $s \in \R$;
\item \label{Theorem 1.2.iii} $\lambda \in \bgset{\lambda_k : k \ge 1}$, $G(x,s) \le 0$ for a.a. $x \in \Omega$ and all $s \in \R$, and $H(x,s) \le 0$ for a.a. $x \in \bdry{\Omega}$ and all $s \in \R$.
\end{enumroman}
\end{theorem}

In \ref{Theorem 1.2.i} it is only assumed that $\lambda$ is not an eigenvalue from the particular sequence $(\lambda_k)_k$, so we have the following corollary.

\begin{corollary}
Assume $(g_1)$--$(g_3)$ and $(h_1)$--$(h_3)$. Then problem \eqref{1.1} has a nontrivial solution $u \in E$ for a.a.\! $\lambda \in \R$.
\end{corollary}

Weak solutions of problem \eqref{1.1} coincide with critical points of the $C^1$-functional
\begin{equation} \label{1.3}
\Phi(u) = \frac{1}{p}\, \norm{u}^p - \frac{\lambda}{p}\, \norm[p,f]{u}^p - \int_\Omega G(x,u)\, dx - \int_{\bdry{\Omega}} H(x,u)\, dS, \quad u \in E.
\end{equation}
Theorem~\ref{Theorem 1.1} was proved in Autuori et al.\! \cite{AuPuVa} using a linking argument as in Degiovanni and Lancelotti \cite{MR2371112}.
The proof of Theorem~\ref{Theorem 1.2} in this paper is based on the Morse theory as in Perera \cite{MR1998432} and Degiovanni et al.\! \cite{MR2661274}.

\section{Preliminaries}

Let $I(u) = \norm{u}^p$, $J(u) = \norm[p,f]{u}^p$, $u \in E$. Then the eigenfunctions of the problem \eqref{1.2} on the manifold
\[
\M = \bgset{u \in E : J(u) = 1}
\]
and the corresponding eigenvalues coincide with the critical points and the critical values, respectively, of the constrained functional $\widetilde{I} := \restr{I}{\M}$. It was shown in Autuori et al.\! \cite[Proposition 3.4]{AuPuVa} that $\widetilde{I}$ satisfies the \PS{} condition, so we can define an increasing and unbounded sequence of eigenvalues by a minimax scheme. The standard scheme based on the Krasnosel$'$ski\u\i's genus does not provide sufficient Morse theoretical information to prove Theorem \ref{Theorem 1.2}, so we use the cohomological index as in \cite{AuPuVa}. Eigenvalues of the $p$-Laplacian based on a cohomological index were first introduced in Perera \cite{MR1998432} for bounded domains (see also Perera and Szulkin \cite{MR2153141}).

Let us recall the definition of the $\Z_2$-cohomological index of Fadell and Rabinowitz \cite{MR57:17677}. For a symmetric subset $M$ of $E \setminus \set{0}$, let $\overline{M} = M/\Z_2$ be the quotient space of $M$ with each $u$ and $-u$ identified, let $f : \overline{M} \to \RP^\infty$ be the classifying map of $\overline{M}$, and let $f^\ast : \mathcal H^\ast(\RP^\infty) \to \mathcal H^\ast(\overline{M})$ be the induced homomorphism of the Alexander-Spanier cohomology rings. Then the cohomological index of $M$ is defined by
\[
i(M) = \begin{cases}
\sup\, \bgset{m \ge 1 : f^\ast(\omega^{m-1}) \ne 0}, & M \ne \emptyset\\[5pt]
0, & M = \emptyset,
\end{cases}
\]
where $\omega \in \mathcal H^1(\RP^\infty)$ is the generator of the polynomial ring $\mathcal H^\ast(\RP^\infty) = \Z_2[\omega]$. For example, the classifying map of the unit sphere $S^{m-1}$ in $\R^m$, $m \ge 1$, is the inclusion $\RP^{m-1} \incl \RP^\infty$, which induces isomorphisms on $\mathcal H^\kappa$ for $\kappa \le m-1$, so $i(S^{m-1}) = m$.

Let $\F$ denote the class of symmetric subsets of $\M$ and set
\[
\lambda_k := \inf_{\substack{M \in \F\\[1pt]
i(M) \ge k}}\, \sup_{u \in M}\, \widetilde{I}(u), \quad k \ge 1.
\]
Then $(\lambda_k)_k$ is a sequence of eigenvalues of \eqref{1.2}, $\lambda_k \nearrow \infty$, and if $\lambda_k < \lambda_{k+1}$, then
\begin{equation} \label{2.1}
i\big(\{u \in E : I(u) \le \lambda_k\, J(u)\} \setminus \{0\}\big) = i\big(E \setminus \{u \in E : I(u) \ge \lambda_{k+1}\, J(u)\}\big) = k
\end{equation}
(see, e.g., Perera et al.\! \cite[Propositions 3.52 and 3.53]{MR2640827}).

Recall that the cohomological critical groups at $0$ of the functional $\Phi$, defined in \eqref{1.3},  are given by
\begin{equation} \label{2.2}
C^\kappa(\Phi,0) = \mathcal H^\kappa(\Phi^0 \cap U,\Phi^0 \cap U \setminus \set{0}), \quad \kappa \ge 0,
\end{equation}
where $\Phi^0 = \bgset{u \in E : \Phi(u) \le 0}$, $U$ is any neighborhood of $0$, and $\mathcal H$ denotes Alexander-Spanier cohomology with $\Z_2$-coefficients. They are independent of $U$ by the excision property.

In the absence of a direct sum decomposition, the main technical tool we use to get an estimate of the critical groups at zero is the notion of a cohomological local splitting
introduced in Perera et al.\! \cite{MR2640827}, which is a variant of the homological local linking of Perera \cite{MR1700283} (see also Li and Willem \cite{MR96a:58045}). The following slightly different form of this notion was given in Degiovanni et al.\! \cite{MR2661274}.

\begin{definition}
We say that $\Phi \in C^1(E,\R)$ has a {\em cohomological local splitting near $0$ in dimension $k \ge 1$} if there are symmetric cones $E_\pm \subset E$ with $E_+ \cap E_- = \set{0}$ and $\rho > 0$ such that
\begin{equation} \label{2.3}
i(E_- \setminus \{0\}) = i(E \setminus E_+) = k
\end{equation}
and
\begin{equation} \label{2.4}
\Phi(u) \le \Phi(0) \quad \mbox{for all }u \in B_\rho \cap E_-, \qquad \Phi(u) \ge \Phi(0) \quad \mbox{for all }u \in B_\rho \cap E_+,
\end{equation}
where $B_\rho = \bgset{u \in E : \norm{u} \le \rho}$.
\end{definition}

\begin{proposition}[{\cite[Proposition 2.1]{MR2661274}}] \label{Proposition 2.4}
If $\Phi \in C^1(E,\R)$ has a cohomological local splitting near $0$ in dimension $k$ and $0$ is an isolated critical point of $\Phi$, then $C^k(\Phi,0) \ne 0$.
\end{proposition}

\section{Proof of Theorem \ref{Theorem 1.2}}

First we prove some lemmas. By Autuori et al.\! \cite[Lemma 4.1]{AuPuVa},
\begin{equation} \label{3.1}
\Phi(u) = \frac{1}{p}\, \norm{u}^p - \frac{\lambda}{p}\, \norm[p,f]{u}^p + \o(\norm{u}^p) \as \norm{u} \to 0.
\end{equation}

\begin{lemma} \label{Lemma 3.1}
If $0$ is an isolated critical point of $\Phi$, then $C^\kappa(\Phi,0) \isom \delta_{\kappa0}\, \Z_2$ for all $\kappa\ge0$ in the following cases:
\begin{enumroman}
\item \label{Lemma 3.1.i} $\lambda < \lambda_1$;
\item \label{Lemma 3.1.ii} $\lambda = \lambda_1$, $G(x,s) \le 0$ for a.a. $x \in \Omega$ and all $s \in \R$, and $H(x,s) \le 0$ for a.a. $x \in \bdry{\Omega}$ and all $s \in \R$.
\end{enumroman}
\end{lemma}

\begin{proof}
We show that $\Phi$ has a local minimizer at $0$. Since $\lambda_1 = \inf_{u \in \M}\, \widetilde{I}(u)$,
\begin{equation} \label{3.2}
\norm{u}^p \ge \lambda_1 \norm[p,f]{u}^p \quad\text{for all }u \in E.
\end{equation}

\ref{Lemma 3.1.i} For sufficiently small $\rho > 0$,
\[
\Phi(u) \ge \left(1 - \frac{\max \set{\lambda,0}}{\lambda_1} + \o(1)\right) \frac{\norm{u}^p}{p} \ge 0 \quad \text{for all }u \in B_\rho
\]
by \eqref{3.1} and \eqref{3.2}.

\ref{Lemma 3.1.ii} We have
\[
\Phi(u) \ge - \int_\Omega G(x,u)\, dx - \int_{\bdry{\Omega}} H(x,u)\, dS \ge 0 \quad \text{for all }u \in E. \QED
\]
\end{proof}

\begin{lemma} \label{Lemma 3.2}
If $0$ is an isolated critical point of $\Phi$, then $C^k(\Phi,0) \ne 0$ in the following cases:
\begin{enumroman}
\item \label{Lemma 3.2.i} $\lambda_k < \lambda < \lambda_{k+1}$;
\item \label{Lemma 3.2.ii} $\lambda = \lambda_k < \lambda_{k+1}$, $G(x,s) \ge 0$ for a.a. $x \in \Omega$ and all $s \in \R$, and $H(x,s) \ge 0$ for a.a. $x \in \bdry{\Omega}$ and all $s \in \R$;
\item \label{Lemma 3.2.iii} $\lambda_k < \lambda_{k+1} = \lambda$, $G(x,s) \le 0$ for a.a. $x \in \Omega$ and all $s \in \R$, and $H(x,s) \le 0$ for a.a. $x \in \bdry{\Omega}$ and all $s \in \R$.
\end{enumroman}
\end{lemma}

\begin{proof}
We show that $\Phi$ has a cohomological local splitting near $0$ in dimension $k$ and apply Proposition \ref{Proposition 2.4}. Let
\[
E_- = \bgset{u \in E : \norm{u}^p \le \lambda_k \norm[p,f]{u}^p}, \qquad E_+ = \bgset{u \in E : \norm{u}^p \ge \lambda_{k+1} \norm[p,f]{u}^p}.
\]
Then \eqref{2.3} holds by \eqref{2.1}, so it only remains to show that \eqref{2.4} holds for sufficiently small $\rho > 0$.

\ref{Lemma 3.2.i} For sufficiently small $\rho > 0$,
\[
\Phi(u) \le - \left(\frac{\lambda}{\lambda_k} - 1 + \o(1)\right) \frac{\norm{u}^p}{p} \le 0 \quad \text{for all }u \in B_\rho \cap E_-
\]
and
\[
\Phi(u) \ge \left(1 - \frac{\lambda}{\lambda_{k+1}} + \o(1)\right) \frac{\norm{u}^p}{p} \ge 0 \quad \text{for all }u \in B_\rho \cap E_+
\]
by \eqref{3.1}.

\ref{Lemma 3.2.ii} For all $u \in E_-$,
\[
\Phi(u) \le - \int_\Omega G(x,u)\, dx - \int_{\bdry{\Omega}} H(x,u)\, dS \le 0,
\]
and for sufficiently small $\rho > 0$, $\Phi(u) \ge 0$ for all $u \in B_\rho \cap E_+$ as in \ref{Lemma 3.2.i}.

\ref{Lemma 3.2.iii} For sufficiently small $\rho > 0$ we have $\Phi(u) \le 0$ for all $u \in B_\rho \cap E_-$, as in \ref{Lemma 3.2.i}. On the other hand, for all $u \in E_+$,
\[
\Phi(u) \ge - \int_\Omega G(x,u)\, dx - \int_{\bdry{\Omega}} H(x,u)\, dS \ge 0. \QED
\]
\end{proof}

\begin{lemma} \label{Lemma 3.3}
There exists $a < 0$ such that $\mathcal H^\kappa(E,\Phi^a) = 0$ for all $\kappa\ge0$, where
$$\Phi^a = \bgset{u \in E : \Phi(u) \le a}.$$
\end{lemma}

\begin{proof} Let us first note that by $(g1)$-$(g2)$ and $(h1)$-$(h2)$ there exists a constant $C>0$ such that
\begin{equation}\label{GH}\begin{gathered}|G(x,s)|\le C w_2(x)|s|^p \quad\text{ for a.a. $x\in\Omega$ and all }s\in\mathbb R,\\
|H(x,s)|\le C w_3(x)|s|^p\quad\text{ for a.a. $x\in\partial\Omega$ and all }s\in\mathbb R.\end{gathered}\end{equation}
Next we show that for all $u\in E$, with $u \ne 0$,
\begin{equation} \label{3.3}\Phi(tu) \to - \infty \as t \to \infty.
\end{equation}
Indeed, fix $u\in E$, with $u \ne 0$. Hence, either $|u|>0$ in a subset $A$ of $\Omega$,
 with meas$(A)>0$, or $|u|>0$ in a subset $\Gamma$ of $\partial\Omega$, with meas$_{N-1}(\Gamma)>0$, being $w>0$ a.e. in $\Omega$ and $\tilde w>0$ a.e. in $\partial\Omega$.\smallskip

\noindent{\it Case $|u|>0$ in $A$, {\rm meas}$(A)>0$.} By $(g2)$ as $|t|\to\infty$
$$\frac{G(x,tu(x))}{|t|^p}= \frac{G(x,tu(x))}{g_0(x)|t|^p|u(x)|^p}\,g_0(x)|u(x)|^p\to\infty\quad\mbox{ a.e. in }A,$$
and $|t|^{-p}|G(x,tu(x))|\le C w_2(x)|u|^p\in L^1(\Omega)$ by \eqref{GH}.
Hence $\lim_{|t|\to\infty}|t|^{-p}\int_\Omega G(x,tu(x))\,dx=\infty$ by the Fatou lemma.\smallskip

\noindent{\it Case $|u|>0$ in $\Gamma$, {\rm
meas}$_{N-1}(\Gamma)>0$.} Similarly, by $(h2)$ as $|t|\to\infty$
$$\frac{H(x,tu(x))}{|t|^p}= \frac{H(x,tu(x))}{h_0(x)|t|^p|u(x)|^p}\,h_0(x)|u(x)|^p\to\infty\quad\mbox{ a.e. in }\Gamma,$$
and $|t|^{-p}|H(x,tu(x))|\le C w_3(x)|u|^p\in L^1(\partial\Omega)$ by \eqref{GH}. Thus
$\lim_{|t|\to\infty}|t|^{-p}\int_{\partial\Omega} H(x,tu(x))\,dS=\infty$ again by the Fatou lemma.

In conclusion we have
$$\frac{\Phi(tu)}{|t|^p}=\frac1p\left(\norm{u}^p - \lambda\norm[p,f]{u}^p\right)-\int_\Omega \frac{G(x,tu(x))}{|t|^{p}}\,dx-\int_{\partial\Omega} \frac{H(x,tu(x))}{|t|^{p}}\,dS\to-\infty$$
as $|t|\to\infty$, along any $u\in E$, with $u \ne 0$. This shows the claim \eqref{3.3}.

Next we prove that there exists $a < 0$ such that
\begin{equation} \label{3.4}
\dualp{\Phi'(u)}{u} < 0 \quad \text{for all }u \in \Phi^a.
\end{equation}
Let us first claim that for all $c\in\mathbb R$
\begin{equation} \label{3.5}
\limsup_{\norm{u}\to\infty,\,\,\Phi(u)\le c}\frac{\dualp{\Phi'(u)}{u}-p\Phi(u)}{\norm{u}^p}<0.$$
\end{equation}
Fix $c\in\mathbb R$. As in Lemma~3.1 of~\cite{MR2661274} we assume by contradiction that there exist sequences $(d_n)_n\subset\mathbb R$ and $(u_n)_n\subset \Phi^c$
such that $d_n\to0$ and
$$\dualp{\Phi'(u_n)}{u_n}-p\Phi(u_n)\ge-d_n\norm{u_n}^p \quad\text{for all }n\in\mathbb N.$$
Clearly $v_n=u_n/\|u_n\|$ is in the unit sphere of
$E$ for all $n$ sufficiently large. Therefore, up to a
subsequence, still denoted for simplicity by $(v_n)_n$, there is
$v\in E$ such that $v_n\rightharpoonup v$ in $E$, $v_n\to v$ in
$L^p(\Omega,w)$, $v_n\to v$ in $L^p(\partial\Omega,\tilde w)$,
$v_n\to v$ a.e. in $\Omega$ and $v_n\to v$ a.e. in
$\partial\Omega$ and $|v_n|\le\phi$ a.e. in $\Omega$ for all $n$, with $\phi\in L^p(\Omega,w_1)\cap L^p(\partial\Omega,w_3)$.

By $(g3)$ and $(h3)$ we have
$$\begin{aligned}-d_n\norm{u_n}^p&\le\dualp{\Phi'(u_n)}{u_n}-p\Phi(u_n)\\
&=\int_\Omega\big[\mu G(x,u_n)-u_ng(x,u_n)\big]dx+\int_{\partial\Omega}\big[\tilde\mu H(x,u_n)-u_nh(x,u_n)\big]dS\\
&\qquad-(\mu-p)\int_\Omega G(x,u_n)dx-(\tilde\mu-p)\int_{\partial\Omega} H(x,u_n)dS\\
&\le\int_\Omega\big(\gamma_0+\gamma_1|u_n|^p\big)dx+\int_{\partial\Omega}\big(\tau_0+\tau_1|u_n|^p\big)dS\\
&\qquad-(\mu-p)\int_\Omega G(x,u_n)dx-(\tilde\mu-p)\int_{\partial\Omega} H(x,u_n)dS.
\end{aligned}$$
Hence
\begin{equation} \label{3.6}
\limsup_n\dfrac{(\mu-p)\displaystyle\int_\Omega G(x,u_n)dx+(\tilde\mu-p)\displaystyle\int_{\partial\Omega} H(x,u_n)dS}{\norm{u_n}^p}<\infty.
\end{equation}
Moreover, by \eqref{GH}
\begin{equation} \label{L}\dfrac{|G(x,u_n)|}{\norm{u_n}^p}=\dfrac{|G(x,\norm{u_n}v_n)|}{\norm{u_n}^p}\le C w_2|v_n|^p\le C w_2\phi,\quad\dfrac{|H(x,u_n)|}{\norm{u_n}^p}\le C w_3 \phi\end{equation}
a.e. in $\Omega$ and a.e. on $\partial\Omega$, respectively.
Hence by the Fatou lemma and \eqref{3.6}
$$\begin{aligned}-\infty<(\mu-p)&\int_\Omega\liminf_n\dfrac{G(x,u_n)}{\norm{u_n}^p}\,dx+(\tilde\mu-p)\int_{\partial\Omega}\liminf_n\dfrac{H(x,u_n)}{\norm{u_n}^p}
\,dS\\
&\qquad\le\liminf_n\dfrac{(\mu-p)\displaystyle\int_\Omega G(x,u_n)dx+(\tilde\mu-p)\displaystyle\int_{\partial\Omega}H(x,u_n)dS}{\norm{u_n}^p}
<\infty.\end{aligned}$$
On the other hand, as shown above, by $(g2)$ at any point $x\in\Omega$ at which $v(x)\not=0$
$$\lim_n\frac{G(x,u_n(x))}{\norm{u_n}^p}=\lim_n\frac{G(x,\norm{u_n}v_n(x))}{g_0(x)\norm{u_n}^pv_n(x)}\,g_0(x)v_n(x)=\infty,$$
and similarly by $(h2)$ at any point $x\in\partial\Omega$ at which $v(x)\not=0$
$$\lim_n\frac{H(x,u_n(x))}{\norm{u_n}^p}=\lim_n\frac{H(x,\norm{u_n}v_n(x))}{h_0(x)\norm{u_n}^pv_n(x)}\,h_0(x)v_n(x)=\infty.$$
In conclusion the limit $v$ is trivial in $E$, that is $v=0$ a.e. in $\Omega$ and on $\partial\Omega$ in the sense of the $N-1$ measure on $\partial\Omega$.
Hence by \eqref{L}
\begin{equation}\label{3.7}\lim_n\dfrac{\displaystyle\int_\Omega G(x,u_n)dx+\displaystyle\int_{\partial\Omega}H(x,u_n)dS}{\norm{u_n}^p}=0.\end{equation}
Now, passing to the limit as $n\to\infty$ in both sides of the following inequality
$$\frac1p-\lambda\|v_n\|^p_{p,f}-\dfrac{\displaystyle\int_\Omega G(x,u_n)dx+\displaystyle\int_{\partial\Omega}H(x,u_n)dS}{\norm{u_n}^p}=\dfrac{\Phi(u_n)}{\norm{u_n}^p}\le\frac c{\norm{u_n}^p},$$
we get $1/p\le0$ by \eqref{3.7}. This contradiction completes the proof of the claim~\eqref{3.5}.

Now by \eqref{3.5}, with $c=0$, there exists $b\in\mathbb R$ such that $\dualp{\Phi'(u)}{u}-p\Phi(u)\le b$ for all $u\in \Phi^0$. In particular, there exists $a<0$ such that~\eqref{3.4} holds.

Fix $u\in E$, with $u\ne 0$. Therefore $\Phi(tu) \le a$ for all sufficiently large $t$ by \eqref{3.3}. Set
\[
T(u) = \inf\, \bgset{t \ge 1 : \Phi(tu) \le a}.
\]
Since
\[
\Phi(tu) \le a \implies \frac{d}{dt}\, \big(\Phi(tu)\big) = \dualp{\Phi'(tu)}{u} = \frac{1}{t} \dualp{\Phi'(tu)}{tu} < 0
\]
by \eqref{3.4}, then
\[
\Phi^a = \bgset{tu : u \in E \setminus \set{0},\, t \ge T(u)}
\]
and the map $E \setminus \set{0} \to [1,\infty),\, u \mapsto T(u)$, is continuous. Thus $E \setminus \set{0}$ radially deformation retracts to $\Phi^a$ via
\[
(E \setminus \set{0}) \times [0,1] \to E \setminus \set{0}, \quad (u,t) \mapsto (1 - t)\, u + t\, T(u)\, u,
\]
and hence
\[
\mathcal H^\kappa(E,\Phi^a) \isom \mathcal H^\kappa(E,E \setminus \set{0}) = 0 \quad \text{for all }\kappa\ge0. \QED
\]
\end{proof}

We are now ready to prove our main result.

\begin{proof}[Proof of Theorem~$\ref{Theorem 1.2}$]
Suppose $0$ is the only critical point of $\Phi$. Since $\Phi$ satisfies the \PS{} condition by Autuori et al.\! \cite[Lemma 4.2]{AuPuVa}, then $\Phi^0$ and $\Phi^a$ are deformation retracts of $E$ and $\Phi^0 \setminus \set{0}$, respectively, by the second deformation lemma. Thus, taking $U = E$ in \eqref{2.2} gives
\[
C^\kappa(\Phi,0) = H^\kappa(\Phi^0,\Phi^0 \setminus \set{0}) \isom \mathcal H^\kappa(E,\Phi^a) = 0 \quad \mbox{for all }\kappa \ge 0
\]
by Lemma \ref{Lemma 3.3}. However, $C^\kappa(\Phi,0) \ne 0$ for some $\kappa \ge 0$ by Lemmas~\ref{Lemma 3.1} and~\ref{Lemma 3.2}. This contradiction completes the proof.
\end{proof}

\def\cdprime{$''$}

\end{document}